\documentclass[12pt]{article}%
\usepackage{amsmath}
\usepackage{amsfonts}
\usepackage{amssymb}
\usepackage{graphicx}%
\providecommand{\U}[1]{\protect\rule{.1in}{.1in}}
\newtheorem{theorem}{Theorem}

\newtheorem{definition}[theorem]{Definition}
\newtheorem{example}[theorem]{Example}

\newtheorem{lemma}[theorem]{Lemma}

\newtheorem{proposition}[theorem]{Proposition}
\newtheorem{remark}[theorem]{Remark}

\newenvironment{proof}[1][Proof]{\noindent\textbf{#1.} }{\ \rule{0.5em}{0.5em}}
\begin{document}

\title{On compact operators between lattice normed spaces }
\author{Youssef Azouzi\quad and\quad Mohamed Amine Ben Amor\\{\small Research Laboratory of Algebra, Topology, Arithmetic, and Order}\\{\small Department of Mathematics}\\{\small Faculty of Mathematical, Physical and Natural Sciences of Tunis}\\{\small Tunis-El Manar University, 2092-El Manar, Tunisia}}
\date{}
\maketitle

\begin{abstract}
In this paper we continue the study of compact-like operators in lattice
normed spaces started recently by Aydin, Emelyanov, Erkur\c{s}un \"{O}zcand
and Marabeh. We show among others, that every $p$-compact operator between
lattice normed spaces is $p$-bounded. The paper contains answers of almost all
questions asked by these authors.

\end{abstract}

\section{Introduction}

In \cite{L-287}, the authors introduced a new notion of compact operators in
Lattice-normed spaces and studied some of their properties. These operators
act on spaces equipped with vector valued norms taking their values in some
vector lattices. Recall that an operator from a normed space $X$ to a normed
space $Y$ is said to be compact if the image of every norm bounded sequence
$\left(  x_{n}\right)  $ in $X$ has a norm convergent subsequence. This notion
has been generalized in the setting of lattice normed spaces giving rise to
two new notions: sequentially $p$-compactness and $p$-compactness ($p$
referred to the vector valued norm). Notice that these notions coincide in the
classical case of Banach spaces. In general setting with vector lattice valued
norms boundedness and convergence are considered with respect to these
`norms'. Also as notions of relatively uniform convergence and almost order
boundedness have been generalized, new properties for the operator are
considered like semicompactness. Recall that if $(E,p,V)$ and $(F,q,W)$ are
Lattice-normed spaces and $T$ is a linear operator from $E$ to $F$, then $T$
is said to be $p$\textit{-compact} (respectively, $rp$\textit{-compact}) if
for every $p$-bounded net $\left(  x_{\alpha}\right)  $ in $E$, there is a
subnet $\left(  T{x}_{\varphi\left(  \beta\right)  }\right)  $ that
$p$-converges (respectively, $rp$-converges) to some $y\in F$. The operator is
said to be \textit{sequentially }$p$\textit{-compact }if nets and subnets are
replaced by sequences and subequences above. In this paper we prove some new
results in this direction. Namely we show that every $p$-compact operator is
$p$-bounded. As a consequence we get that every $rp$-compact is $p$-bounded.
Also we give an example of sequentially $p$-compact operator which fails to be
$p$-bounded. As a consequence we deduce that a sequentially $p$-compact need
not be $p$-compact. In fact these two notions are totally independent. Example
of $p$-compact operators that fail to be sequentially $p$-compact is given. As
mentioned above the study of $p$-compact operators between lattice normed
spaces was started in \cite{L-287}. That paper contains several new results
but also some open questions. Almost all these questions will be answered in
our paper.

\section{Preliminaries}

The goal of this section is to introduce some basic definitions and facts. For
general informations on vector lattices, Banach spaces and lattice-normed
spaces, the reader is referred to the classical monographs \cite{b-240} and
\cite{b-2474}.

Consider a vector space $E$ and a real Archimedean vector lattice $V$. A map
$p:E\rightarrow V$ is called a \textit{vector norm} if it satisfies the
following axioms:

\begin{enumerate}
\item[1)] $p(x)\geq0;$\thinspace\thinspace\ $p(x)=0\Leftrightarrow
x=0$;\thinspace\thinspace$(x\in E)$.

\item[2)] $p(x_{1}+x_{2})\leq p(x_{1})+p(x_{2});\,\, ( x_{1},x_{2}\in E)$.

\item[3)] $p(\lambda x)=|\lambda| p(x);\,\, (\lambda\in\mathbb{R},\,x\in E)$.
\end{enumerate}

A triple $(E,p,V)$ is a \textit{lattice-normed space} if $p(.)$ is a
$V$-valued vector norm in the vector space $E$. When the space $E$ is itself a
vector lattice the triple $\left(  E,p,V\right)  $ is called a
\textit{lattice-normed vector lattice.} A set $M\subset E$ is called
$p$-\textit{bounded } if $p\left(  M\right)  \subset\lbrack-e,e]$ for some
$e\in V_{+}.$ A subset $M$ of a \textit{lattice-normed vector lattice}
$\left(  E,p,V\right)  $ is called $p$\textit{-almost order bounded} if, for
any $w\in V_{+}$, there is $x_{w}\in E_{+}$ such that $p((|x|-x_{w}%
)^{+})=p(|x|-x_{w}\wedge|x|)\leq w$ for any $x\in M$.

Let $(x_{\alpha})_{\alpha\in\Delta}$ be a net in a \textit{lattice-normed
space} $(E,p,V)$. We say that $(x_{\alpha})_{\alpha\in\Delta}$ is
$p$\textit{-convergent} to an element $x\in E$ and write $x_{\alpha}%
\overset{p}{\longrightarrow}x,$ if there exists a decreasing net $(e_{\gamma
})_{\gamma\in\Gamma}$ in $V$ such that $\inf_{\gamma\in\Gamma}(e_{\gamma})=0$
and for every $\gamma\in\Gamma$ there is an index $\alpha(\gamma)\in\Delta$
such that $p(v-v_{\alpha})\leq e_{\gamma}$ for all $\alpha\geq\alpha(\gamma)$.
Notice that if $V$ is Dedekind complete, the dominating net $\left(
e_{\gamma}\right)  $ may be chosen over the same index set as the original
net. We say that $\left(  x_{\alpha}\right)  $ is $p$\textit{-unbounded
convergent} to $x$ (or for short, $up$\textit{-convergent} to $x$) if
$|x_{\alpha}-x|\wedge u\overset{p}{\longrightarrow}0$ for all $u\in V_{+}$. It
is said to be \textit{relatively uniformly }$p$\textit{-convergent} to $x\in
X$ (written as, $x_{\alpha}\overset{rp}{\longrightarrow}x)$ if there is $e\in
E_{+}$ such that for any $\varepsilon>0$, there is $\alpha_{\varepsilon}$
satisfying $p(x_{\alpha}-x)\leq\varepsilon e$ for all $\alpha\geq
\alpha_{\varepsilon}$.

When $E=V$ and $p$ is the absolute value in $E,$ the $p$-convergence is the
order convergence, the $up$-convergence is the unbounded order convergence,
and the $rp$-convergence is the relatively uniformly convergence. We refer to
\cite{b-2800} and \cite{L-65} for the basic facts about nets in topological
spaces and vector lattices respectively.\textrm{ }We will use
\cite{b-2474,b-1986} as unique source for unexplained terminology in
Lattice-Normed Spaces. Since the most part of this paper is devoted to answer
several open questions in \cite{L-287}, the reader must have that paper handy,
from which we recall some definitions.

\begin{definition}
Let $X,\ Y$ be two \textit{lattice-normed space}s and $T\in L(X,Y)$. Then

\begin{enumerate}
\item $T$ is called $p$\textit{-compact} if, for any $p$-bounded net $\left(
x_{\alpha}\right)  $ in $X$, there is a subnet $x_{\alpha_{\beta}}$ such that
$Tx_{\alpha_{\beta}}\overset{p}{\longrightarrow}y$ in $Y$ for some $y\in Y$.

\item $T$ is called \textit{sequentially }$p$\textit{-compact} if, for any
$p$-bounded sequence $x_{n}$ in $X$, there is a subsequence $\left(  x_{n_{k}%
}\right)  $ such that $Tx_{n_{k}}\overset{p}{\longrightarrow}y$ in $Y$ for
some $y\in Y$.

\item $T$ is called $p$-semicompact if, for any $p$-bounded set $A$ in $X$,
the set $T(A)$ is $p$-almost order bounded in $Y$.
\end{enumerate}
\end{definition}

\section{$p$-compact operators are $p$-bounded}

It is well known that compact operators between Banach spaces are bounded.
This result remains valid for general situation of $p$-compact operators as it
will be shown in our first result, which answers positively Question 2 in
\cite{L-287}.

\begin{theorem}
\label{B}Every $p$-compact operator between two Lattice-normed spaces is $p$-bounded.
\end{theorem}

\begin{proof}
Assume, by contradiction, that there exists a $p$-compact operator
$T:(E,p,V)\longrightarrow(F,q,W)$ which is not $p$-bounded. Then there exists
a $p$-bounded subset $A$ of $E$ such that $T\left(  A\right)  $ is not
$q$-bounded. So, for every $u\in W^{+}$ there exists some $x_{u}\in A$
satisfying $q(T(x_{u}))\not \leq u$. Since the net $(x_{u})_{u\in W^{+}}$ is
$p$-bounded there is a subnet $\left(  y_{v}=x_{\varphi\left(  v\right)
}\right)  _{v\in\Gamma}$ and an element $f\in F$ such that $\left(
Ty_{v}\right)  $ $\overset{q}{\longrightarrow}f$. It follows that the net
$\left(  Ty_{v}\right)  $ has a $q$-bounded tail, which means that for some
$v_{0}$ in $\Gamma$ and some $w\in W^{+}$ we have,%
\begin{equation}
q\left(  Tx_{\varphi\left(  v\right)  }\right)  \leq w,\qquad\text{for }v\geq
v_{0}. \label{I1}%
\end{equation}
Pick $v_{1}$ in $\Gamma$ such that $\varphi(v)\geq w$ for all $v\geq v_{1}$.
It follows that for $v\geq v_{0}\vee v_{1},$ we have $q(Tx_{\varphi\left(
v\right)  })\not \leq \varphi(v)$ and so%
\[
q(Tx_{\varphi(v)})\not \leq w,
\]
which is a contradiction with \ref{I1}. and the proof comes to its end.
\end{proof}

The following lemma, which connects unbounded order convergence with pointwise
convergence, is a known fact, although a quick proof is included for the sake
of completeness.

\begin{lemma}
\label{fx}Let $E=\mathbb{R}^{X}$ be the Riesz space of all real-valued
functions defined on a nonempty set $X.$ The following statements are equivalent:

\begin{enumerate}
\item[\textrm{(i)}] The net $\left(  f_{\alpha}\right)  _{\alpha\in A}$ is
$uo$-convergent in $E.$

\item[\textrm{(ii)}] for every $x\in X,$ the net $\left(  f_{\alpha}\left(
x\right)  \right)  _{\alpha\in A}$ is convergent in $\mathbb{R}.$
\end{enumerate}
\end{lemma}

\begin{proof}
(i) $\Longrightarrow$ (ii) Assume that $f_{\alpha}\overset{uo}{\longrightarrow
}f$ in the Dedekind complete Riesz space $E.$ Then there is a net $\left(
g_{\alpha}\right)  _{\alpha\in A}$ which decreases to $0$ and for some
$\alpha_{0}$ we have%
\begin{equation}
\left\vert f_{\alpha}-f\right\vert \wedge1\leq g_{\alpha}\text{ for all
}\alpha\geq\alpha_{0}. \label{1}%
\end{equation}
Since $\left(  g_{\alpha}\left(  x\right)  \right)  $ decreases to $0$ for
every $x\in X,$ it follows easily from \ref{1} that $f_{\alpha}\left(
x\right)  -f\left(  x\right)  $ converges to $0,$ as desired.

(ii) $\Longrightarrow$ (i) Assume now that $f_{\alpha}$ converges simply to
some $f\in E$ and let $h\in E^{+}.$ Define a net $\left(  g_{\alpha}\right)  $
by putting%
\[
g_{\alpha}\left(  x\right)  =\sup\limits_{\beta\geq\alpha}\left(  \left\vert
f_{\beta}-f\right\vert \wedge h\right)  \left(  x\right)  ,\text{ }x\in X.
\]
it is clear that $g_{\alpha}$ decreases to $0$ and $\left\vert f_{\alpha
}-f\right\vert \wedge h\leq g_{\alpha}.$ This shows that $f_{\alpha}%
\overset{uo}{\longrightarrow}f$ and we are done.
\end{proof}

Consider the Riesz space $F$ of all bounded real valued functions defined on
the real line with countable support and denote by $E$ the direct sum
$\mathbb{R}\mathbf{1}\oplus F,$ where $\mathbf{1}$ denotes the constant
function taking the value $1.$ This example will be of great interest for us.
The following lemma establishes some of its properties. Recall that a vector
sublattice $Y$ of a vector lattice $X$ is said to be \textit{regular }if every
subset in $Y$ having a supremum in $Y$ has also a supremum in $X$ and these
suprema coincide. For more information about this notion and nice
characterizations of it via unbounded order convergence the reader is referred
to \cite{L-65}.

\begin{lemma}
\label{L2}The space $E$ introduced above has the following properties.

\begin{enumerate}
\item[(i)] $E$ is a regular vector sublattice of $\mathbb{R}^{\mathbb{R}}.$

\item[(ii)] $E$ is Dedekind $\sigma$-complete but not Dedekind complete.
\end{enumerate}
\end{lemma}

\begin{proof}
(i) It is clear that $E$ is a vector sublattice of $\mathbb{R}^{\mathbb{R}}.$
To show that it is regular assume that $\left(  g_{\alpha}\right)  _{\alpha\in
A}$ is a net in $E$ satisfying $g_{\alpha}\downarrow0$ in $E.$ Let
$g=\inf\limits_{\alpha}g_{\alpha}$ in $\mathbb{R}^{\mathbb{R}}$ and
$x\in\mathbb{R}.$ Then $h=g\left(  x\right)  \mathbf{1}_{\left\{  x\right\}
}\in E$ and $0\leq h\leq g_{\alpha}$ for all $\alpha.$ This implies that $h=0$
and then $g\left(  x\right)  =0.$ Hence $g=0$ and the regularity is proved.

(ii) Let $\left(  g_{n}\right)  $ be an order bounded sequence in $E$ and
write $g_{n}=\lambda_{n}+f_{n},$ with $\lambda_{n}\in\mathbb{R}$ and $f_{n}\in
F.$ Let $\Omega$ be the union of the supports of $f_{n},$ then $\Omega$ is
countable. Let $g$ be the supremum of $\left(  g_{n}\right)  $ in
$\mathbb{R}^{\mathbb{R}},$ that is,%
\[
g\left(  a\right)  =\sup g_{n}\left(  a\right)  ,\text{ for all }%
a\in\mathbb{R}.
\]
It will be sufficient to show that $g\in E.$ To this end observe that
$g\left(  x\right)  =\alpha:=\sup\alpha_{n}$ for all $x\in\mathbb{R}%
\backslash\Omega.$ Now put $f=\left(  g-\alpha\right)  \mathbf{1}_{\Omega}.$
Then $f\in F$ and $g=\alpha+f\in E$ as required. Next we show that $E$ is not
Dedekind complete. Consider the net $\left(  g_{x}\right)  _{x\in\left[
0,1\right]  }$ in $E$ defined by $g_{x}=x\mathbf{1}_{\left\{  x\right\}  }.$
It is a bounded net in $E$ and its supremum in $\mathbb{R}^{\mathbb{R}}$ does
not belong to $E.$ As $E$ is regular in $\mathbb{R}^{\mathbb{R}}$ this net can
not have a supremum in $E.$
\end{proof}

\begin{remark}
Consider the following operator:
\[
T:L_{1}\left[  0,1\right]  \longrightarrow c_{0};\text{ }\qquad f\longmapsto
Tf=\left(  \int_{0}^{1}f\left(  t\right)  \sin ntdt\right)  _{n\geq1}.
\]
It is mentioned in \cite{b-240}, that $T$ is not order bounded; it is perhaps
more convenient to consider the same operator defined on $L_{1}\left[
0,2\pi\right]  .$ In this case if we define $u_{n}$ by $u\left(  t\right)
=\sin nt$ for $t\in\left[  0,2\pi\right]  ,$ then $\left\vert u_{n}\right\vert
\leq1,$ however $\left(  Tu_{n}\right)  =\left(  e_{n}\right)  $ is not
bounded in $c_{0},$ where $\left(  e_{n}\right)  $ denotes the standard basis
of $c_{0}.$ This statement implies also that $T$ is not sequentially order
compact. Because $\left(  e_{n}\right)  $ has no order bounded subsequence, it
follows that $\left(  Tu_{n}\right)  $ can not admit an order convergent
subsequence. So the statement made in \cite{L-287} that $T$ is $p$-bounded is
not correct.
\end{remark}

The above example is presented in \cite{L-287} to show that sequentially
$p$-compact operators need not be $p$-bounded. Although the operator given in
that example fails to be sequentially $p$-compact, the assertion that
sequentially $p$-compact operators need not be $p$-bounded is true. This will
be shown in our next example.

\begin{example}
\label{Right}Consider the Riesz spaces $E$ and $F$ defined just before Lemma
\ref{L2} and let $T$ be the projection defined on $E$ with range $F$ and
kernel $\mathbb{R}\mathbf{1}.$ We claim that $T$ is sequentially order
compact, but not order bounded. Let $\left(  f_{n}\right)  $ be an order
bounded sequence in $E.$ Then $\left\vert f_{n}\right\vert \leq\lambda$ for
some real $\lambda>0$ and for all $n.$ Write $f_{n}=g_{n}+\lambda_{n}$ with
$\lambda_{n}$ real and $g_{n}\in F$ and observe that $\left\vert
g_{n}\right\vert \leq2\lambda$ for all $n.$ We have also $\left\vert
g_{n}\right\vert \leq2\lambda\mathbf{1}_{A}\in F$ where $A$ is the union of
the supports of $g_{n},$ $n=1,2,...$ A standard diagonal process yields a
subsequence $\left(  g_{k_{n}}\right)  $ of $\left(  g_{n}\right)  $ which
converges pointwise on $A$ and then on $\mathbb{R}$ since all functions
$g_{n_{k}}$vanish on $\mathbb{R}\backslash A.$ Hence $\left(  g_{k_{n}%
}\right)  $ is uo-convergent in $\mathbb{R}^{\mathbb{R}}.$ As $\left(
g_{k_{n}}\right)  $ is order bounded this implies that $\left(  g_{k_{n}%
}\right)  $ is order convergent in $\mathbb{R}^{\mathbb{R}}.$ Observe moreover
that $\sup\limits_{p\geq n}g_{k_{n}}$ belongs to $F,$ which shows that
$\left(  g_{k_{n}}\right)  $ is order convergent in $F.$ The fact that $T$ is
not order bounded is more obvious: it is clear that the image of the net
$\left(  \mathbf{1}_{\left\{  x\right\}  }\right)  _{x\in\left[  0,1\right]
}$ by $T$ is not order bounded in $F.$
\end{example}

As an immediate consequence of Theorem \ref{B} and Example \ref{Right} we
deduce that sequentially $p$-compactness does not imply $p$-compactness. At
this stage one might expect that the converse is true. Does $p$-compactness
imply sequentially $p$-compactness? This is an open question left in
\cite{L-287}. Unfortunately the answer is again negative.

\begin{example}
Let $X$ be the set of all strictly increasing maps from $\mathbb{N}$ to
$\mathbb{N}$ and $E=\mathbb{R}^{X}$ be the space of all real-valued functions
defined on $X,$ equipped with the product topology.

\begin{enumerate}
\item First we will prove that the identity map, $\mathfrak{I},$ is a
$p$-compact operator on the lattice-normed space $(E,|\ |,E)$. To this aim,
pick a $p$-bounded net $\left(  f_{\alpha}\right)  _{\alpha\in A}$ in $E$,
that is, $|f_{\alpha}|\leq f$ for some $f\in E^{+}$ and for every $\alpha\in
A$. It follows that
\[
f_{\alpha}\in\prod_{x\in X}[-f(x),f(x)].
\]
Notice that the space $\prod_{x\in X}[-f(x),f(x)],$ equipped with the product
topology, is compact by Tychonoff's Theorem. Thus $\left(  f_{\alpha}\right)
$ has a convergent subnet $\left(  g_{\beta}\right)  _{\beta\in B}$ in
$\prod_{x\in X}[-f(x),f(x)]$ to some $g.$ This means that%
\[
g_{\beta}(x)\longrightarrow g(x)\ \text{for all }x\in X.
\]
According to Lemma \ref{fx}, $g_{\beta}$ is $uo$-convergent to $g$ in $E.$
Since bounded uo-convergent nets are order convergent, we have that $g_{\beta
}\overset{o}{\longrightarrow}g$. This proves that $\mathfrak{I}$ is a
$p$-compact operator.

\item We prove now that $\mathfrak{I}$ is not sequentially $p$-compact. Let
$\left(  \varphi_{n}\right)  $ be a sequence in $\{-1,\ 1\}^{X}$ which has no
convergent subsequence (see Example 3.3.22 in \cite{b-348}). This sequence is
order bounded in $E$ and every subsequence $\left(  \psi_{n}\right)  $ of
$\left(  \varphi_{n}\right)  $ does not converges in $\{-1,\ 1\}^{X},$ that
is, for some $x\in X,$ $\psi_{n}\left(  x\right)  $ diverges. According to
Lemma \ref{fx}, $\left(  \psi_{n}\right)  $ is not uo-convergent in $E.$ Since
$\left(  \psi_{n}\right)  $ is order bounded it does not converge in order.
This finishes the proof.
\end{enumerate}
\end{example}

In classical theory of Banach spaces the identity map is compact if and only
if the space is finite-dimensional. In contrast of this the situation is not
clear in general case. We already have seen an example of infinite-dimensional
space on which the identity map is $p$-compact. This question has been
investigated in \cite{L-287} where the authors showed that $I_{L_{1}\left[
0,1\right]  }$ fails to be compact however, $I_{\ell_{1}}$ is $p$-compact. In
the next example we show that $I_{L_{\infty}\left[  0,1\right]  }$ is not
$p$-compact, answering a question asked in \cite{L-287}.

\begin{example}
The identity operator $I$ on the lattice normed space $(L_{\infty
}[0,1],|\ .\ |,L_{\infty}[0,1])$ is neither $p$-compact nor sequentially
$p$-compact. To this end, consider the sequence of Rademacher function given
by :
\[%
\begin{array}
[c]{ccccl}%
r_{n} & : & [0,1] & \longrightarrow & \mathbb{R}\\
\  & \  & \ t & \longmapsto & sgn\left(  \sin(2^{n}\pi t)\right)
\end{array}
\]
for all $n\in\mathbb{N},$ which is order bounded since $|r_{n}|=1.$ Suppose
now that $\left(  r_{n}\right)  $ has an order convergent subnet $\left(
r_{n_{\alpha}}\right)  _{\alpha\in\Gamma}.$ Then $r_{n_{\alpha}}\overset
{o}{\rightarrow}r$ for some $r\in L_{\infty}\left[  0,1\right]  .$ Let
$\alpha\in A.$ For every $\beta>\alpha,$ $\int_{0}^{1}r_{n_{\alpha}%
}r_{n_{\beta}}d\mu=0.$ On the other hand $\left(  r_{n_{\alpha}}r_{n_{\beta}%
}\right)  _{\beta}$ converges in order to $r_{n_{\alpha}}r$ in $L_{\infty
}\left[  0,1\right]  $ and then in $L_{1}\left[  0,1\right]  .$ Since the
integral is order continuous, we deduce that
\[
\int_{0}^{1}r_{n_{\alpha}}rd\mu=0.
\]
This equality holds for every $\alpha\in A,$ and a similar argument leads to%
\[
\int_{0}^{1}r^{2}d\mu=0,
\]
which is a contradiction since $\left\vert r\right\vert =1,$ and the claim is
now proved.
\end{example}

\section{Semicompact operators}

The notion of semicompact operators has been introduced by Zaanen in
\cite{b-1087} and extended in the framework of lattice normed spaces in
\cite{L-287}.

Let $(X,p,E)$ be a lattice normed space and $(Y,q,F)$ be an lattice normed
vector lattice. A linear operator $T:$ $X\rightarrow Y$ is called
$p$\textit{-semicompact} if it maps $p$-bounded sets in $X$ to $p$-almost
order bounded sets in $Y$. We recall that a subset $B$ of $Y$ is said to be
$p$-almost order bounded if for any $w\in F_{+}$, there is $y_{w}\in Y$ such that%

\[
q((|y|-y_{w})^{+})=q(|y|-y_{w}\wedge|y|)\leq w\ \text{for all }y\in B.
\]

Semicompact operators from Banach spaces to Banach lattices fail, in general,
to be compact (see \cite{b-240}). This yields trivially that $p$%
-semicompactness does not imply $p$-compactness. However, the converse is true
in the classical case as has been shown in Theorem 5.71 in \cite{b-240}$.$ And
one can expect to extend this result in general situation. This is already the
subject of Question 4 in \cite{L-287}. Unfortunately the answer is again
negative. Before stating our counterexample let us recall that every order
bounded operator from a vector lattice $E$ to a Dedekind complete vector
lattice $F$ has a modulus (\cite{b-240}).

\begin{example}
\label{E1}Let $E$ be a Dedekind complete Banach lattice with order continuous
norm and $T$ be a norm-compact operator in $\mathcal{L}(E)$ such that $T$ has
no modulus, and therefore $T$ can not be order bounded. For the existence of
such operator we refer the reader to the Krengel's example in \cite[p
277.]{b-240}. Consider now the following lattice-normed vector spaces
$(E,\Vert.\Vert,\mathbb{R})$ and $\left(  E,p,\mathbb{R}^{2}\right)  ,$ where
$p(x)=%
\begin{pmatrix}
\Vert x\Vert\\
0
\end{pmatrix}
$ for all $x\in E.$ It is straightforward to prove that $T$ is again
$p$-compact operator and we claim that $T$ is not $p$-semicompact. To this end
we we will argue by contradiction and we assume that $T$ is $p$-semicompact.
Fix an element $u\in E^{+}$ and let $w=%
\begin{pmatrix}
0\\
1
\end{pmatrix}
$, then there exists $z_{w}$ such that $p\left(  (T(x)-z_{w})^{+}\right)  \leq
w$ for all $x\in\lbrack-u,u]$, which means that $(T(x)-z_{w})^{+}=0.$ Noting
that this occurs for $x$ and $-x$ we see that%
\[
|T(x)|\leq z_{w}\text{ for all }x\in\left[  -u,u\right]  .
\]
This shows that $T$ is order bounded, a contradiction. and our proof comes to
an end.
\end{example}

A slight modification of the proof of Example \ref{E1} leads to a more general
result. The proof of it will be left for the reader.

\begin{proposition}
Let $\left(  E,p,V\right)  $ be a lattice normed space and $\left(
F,q,W\right)  $ a lattice normed vector lattice. We assume that $q\left(
F\right)  ^{d}$ is not trivial. Then every semicompact operator $T:\left(
E,p,V\right)  \longrightarrow\left(  F,q,W\right)  $ is $p$-bounded as an
operator from $\left(  E,p,V\right)  $ to $\left(  F,\left\vert .\right\vert
,F\right)  .$
\end{proposition}

\section{$rp$-compact operators}

As every $rp$-compact operator between lattice-normed spaces is $p$-compact,
the following result is an immediate consequence of Theorem \ref{B}.

\begin{theorem}
Let $(E,p,V)$ and $(F,q,W)$ be lattice-normed spaces and $T$ be in
$\mathfrak{L}(E,F)$. If $T$ is $rp$-compact then $T$ is $p$-bounded.
\end{theorem}

In the following example we will prove that sequentially $p$-compact operators
need not be $rp$-compact.

\begin{example}
Let $E$ be the Riesz space defined above. We claim that the identity operator
$I:E\rightarrow E$ is sequentially $p$-compact but fails to be $rp$-compact.
Let $\left(  x_{n}\right)  $ be a bounded sequence in $E,$ that is,
$|x_{n}|\leq x$ for some $x\in E^{+}.$ Write $x=\alpha+f,$ and $x_{n}%
=\alpha_{n}+f_{n}$ where $\alpha\in\mathbb{R}^{+}$ and $f\in F$ and
$\alpha_{n}\in\mathbb{R},$ $f_{n}\in F$ for $n=1,2,...$ It is easily seen that
$|\alpha_{n}|\leq\alpha$, $|f_{n}|\leq x+\alpha$. By a standard diagonal
argument there exists a subsequence such that $f_{\varphi(n)}\left(  a\right)
$ converges for every $a\in\mathbb{R}$ and $\alpha_{\varphi\left(  n\right)
}$ converges in $\mathbb{R}.This$ shows that $x_{n}$ converges pointwise on
$\mathbb{R}$ and its limit is clearly in $E.$ By Lemma \ref{fx},
$x_{n}\overset{uo}{\longrightarrow}x$ in $\mathbb{R}^{\mathbb{R}}$ and then
$x_{n}\overset{o}{\longrightarrow}x$ in $\mathbb{R}^{\mathbb{R}}$ as it is an
order bounded sequence. Now using Lemma 27 in \cite{L-444} and Lemma \ref{L2}
we deduce that $x\in E.$ On the other hand, let $\mathcal{F}$ be the
collection of finite subsets of $\mathbb{R}_{+}$ ordered by inclusion and
consider the net $\left(  g_{A}\right)  _{A\in\mathcal{F}}$ where $g_{\alpha
}=\mathbf{1}_{\alpha}$. Then $\left(  g_{\alpha}\right)  $ is order bounded in
$E$ but has no convergent subnet. Since $g_{\alpha}\uparrow\mathbf{1}%
_{\mathbb{R}_{+}}$ in $\mathbb{R}^{\mathbb{R}}$ and $E$ is regular in
$\mathbb{R}^{\mathbb{R}},$ it follows that $\left(  g_{\alpha}\right)  $ is
not order convergent in $E$ and so are all its subnets.
\end{example}


\begin{thebibliography}{9}                                                                                                %


\bibitem {b-240}C.D. Aliprantis and O. Burkinshaw, \textit{Positive
Operators}, Springer, 2006.

\bibitem {L-444}Y. Azouzi, Completeness for vector lattices, J. Math. Anal.
Appl. 472 (2019) 216--230

\bibitem {L-287}A. Ayd\i n, E.Yu. Emelyanov, N. Erkur\c{s}un \"{O}zcand,
M.A.A. Marabeh, Indag. Math. (N.S.) 29 (2018) 633--656.

\bibitem {L-65}N. Gao, V.G. Troitsky, F. Xanthos, Uo-convergence and its
applications to Ces\'{a}ro means in Banach lattices, Isr. J. Math. 220 (2)
(2017) 649--689.

\bibitem {b-2800}J. L. Kelley,\textbf{ }General topology. Courier Dover
Publications, 2017.

\bibitem {b-2474}A.G. Kusraev, Dominated Operators, in: Mathematics and its
Applications, vol. 519, Kluwer Academic Publishers, Dordrecht, 2000.

\bibitem {b-348}V. Runde, A Taste of Topology. Springer, Berlin (2005)

\bibitem {b-1986}B.Z. Vulikh, Introduction To the Theory of Partially Ordered
Spaces, Wolters-Noordhoff Scientifc Publications, Ltd., Groningen, 1967.

\bibitem {b-1087}A. C. Zaanen, Riesz spaces II, North-Holland, Amsterdam, 1983.
\end{thebibliography}
\end{document}